\documentclass[11pt]{amsart}
\usepackage{enumerate}
\usepackage{hyperref}
\usepackage{graphicx}
\usepackage{subfig,amssymb}
\usepackage{setspace}
\usepackage[margin=3cm]{geometry}

\hypersetup{
    pdftitle=   {Improper colouring of graphs with no odd clique minor},
   pdfauthor=  {Dong Yeap Kang, Sang-il Oum}
}
\newtheorem{THM}{Theorem}[section]
\newtheorem*{THM*}{Theorem~\ref{main}}
\newtheorem{LEM}[THM]{Lemma}
\newtheorem{OBS}[THM]{Observation}
\newtheorem{COR}[THM]{Corollary}

\newtheorem{CON}[THM]{Conjecture}
\newtheorem{CLAIM}{Claim}

\theoremstyle{remark}

\newcommand*{\qedblack}{\hfill\ensuremath{\blacksquare}}

\theoremstyle{definition}

\begin{document}
\title{Improper coloring of graphs with no odd clique minor}
\author{Dong Yeap Kang}
\author{Sang-il Oum}
\address{Department of Mathematical Sciences, KAIST, 291 Daehak-ro
  Yuseong-gu Daejeon, 34141 South Korea}
\email{dyk90@kaist.ac.kr}
\email{sangil@kaist.edu}
\thanks{Supported by the National Research Foundation of Korea (NRF) grant funded by the Korea government (MSIT) (No. NRF-2017R1A2B4005020).}
\thanks{The first author has been supported by TJ Park Science Fellowship of POSCO TJ Park Foundation.}
\date{\today}

\begin{abstract}
As a strengthening of Hadwiger's conjecture,
Gerards and Seymour conjectured that every graph with no odd $K_t$ minor is $(t-1)$-colorable.
We prove two weaker variants of this conjecture.
Firstly, we show that for each $t \geq 2$, every graph with no odd $K_t$ minor has a partition of its vertex set into $6t-9$ sets $V_1, \dots, V_{6t-9}$ such that each $V_i$ induces a subgraph of bounded maximum degree. 
Secondly, we prove that for each $t \geq 2$, every graph with no odd $K_t$ minor has a partition of its vertex set into $10t-13$ sets $V_1, \dots, V_{10t-13}$ such that each $V_i$ induces a subgraph with components of bounded size. 
The second theorem improves a result of Kawarabayashi (2008), which states that the vertex set can be partitioned into $496t$ such sets.\\

\textup{2010} \textit{Mathematics Subject Classification}:\:\:Primary: 05C15; Secondary: 05C83
\end{abstract}


\maketitle

\section{Introduction}\label{sec:intro}
Every graph in this paper is finite and simple. For a nonnegative integer $k$, a graph $G$ is (properly) $k$-colorable if there are $k$ pairwise disjoint sets $V_1$, $\dots$, $V_k$  with $V(G) = \bigcup_{i=1}^{k}{V_i}$ such that $V_i$ induces a subgraph of maximum degree $0$ for $1 \leq i \leq k$. 

In 1943, Hadwiger~\cite{Hadwiger1943} proposed the following question, which is called ``Hadwiger's conjecture", one of the deepest conjectures in graph theory. 
For more on this conjecture and its variants, the readers are referred to the recent survey of Seymour~\cite{Seymour2015b}.
\begin{CON}[Hadwiger~\cite{Hadwiger1943}]\label{hadwiger}
For each integer $t \geq 1$, every graph with no $K_t$ minor is $(t-1)$-colorable.
\end{CON}

Robertson, Seymour, and Thomas~\cite{RST1993a} proved that the conjecture is true for $t \leq 6$, but the conjecture remains open for $t \geq 7$.
Kostochka~\cite{Kostochka1982, Kostochka1984} and Thomason~\cite{Thomason1984, Thomason2001} proved that graphs with no $K_t$ minor are  $O(t \sqrt{\log t})$-colorable, by showing that these graphs contain a vertex of degree $O(t \sqrt{\log t})$. It is still open whether every graph with no $K_t$ minor is $ct$-colorable for some $c>0$ independent of $t$.

Gerards and Seymour~(see~\cite[Section 6.5]{JT1995a}) proposed the following odd-minor variant of Hadwiger's conjecture.
\begin{CON}[Gerards and Seymour~{(see~\cite[Section 6.5]{JT1995a})}]\label{oddhadwiger}
For each integer $t \geq 1$, every graph with no odd $K_t$ minor is $(t-1)$-colorable.
\end{CON}
Catlin~\cite{Catlin1979} proved this conjecture for $t=4$, and Guenin~\cite{Guenin05} announced a proof for $t=5$, but the proof has not been written.
Geelen, Gerards, Reed, Seymour, and Vetta~\cite{GGRSV2009} proved that every graph with no odd $K_t$ minor is $O(t \sqrt{\log t})$-colorable. 

A \emph{defective coloring} (see~\cite{CGJ1997,Wood2018}) is a coloring that relaxes the degree condition. A graph $G$ is \emph{$k$-colorable with defect $d$} if there are $k$ pairwise disjoint sets $V_1 , \dots , V_k$ with $V(G) = \bigcup_{i=1}^{k}{V_i}$ such that every $V_i$ induces a subgraph of maximum degree at most $d$. Note that $G$ is $k$-colorable if and only if $G$ is $k$-colorable with defect $0$.
A \emph{clustered coloring} is a coloring that relaxes the size of monochromatic components. A graph $G$ is \emph{$k$-colorable with clustering $M$} if  there are $k$ pairwise disjoint sets $V_1 , \dots , V_k$ with $V(G) = \bigcup_{i=1}^{k}{V_i}$ such that every $V_i$ induces a subgraph having no component with more than $M$ vertices.
For a class $\mathcal C$ of graphs, the \emph{defective chromatic number} of $\mathcal C$ is the minimum $k$ such that for some $d$, all graphs in $\mathcal C$ are $k$-colorable with defect $d$. 
Similarly the \emph{clustered chromatic number} of $\mathcal C$ is the minimum $k$ such that for some $M$, all graphs in $\mathcal C$ are $k$-colorable with clustering $M$. 

\medskip
We present two theorems, both of which
are relaxations of Conjecture~\ref{oddhadwiger} for graphs with no odd $K_t$ minor.
Our first theorem is about defective coloring.

\begin{THM}\label{mainthm}
For each integer $t \geq 2$, there exists an integer $s = s(t)$ such that every graph $G$ with no odd $K_t$ minor is $(6t-9)$-colorable with defect $s$.
\end{THM}

We remark that the number $6t-9$ of colors cannot be reduced to the number less than $t-1$ (see Theorem~\ref{thm:weakhad}). 

Our second theorem is about clustered coloring.
Kawarabayashi~\cite{Kawarabayashi2008} proved  that
the class of graphs with no odd $K_t$ minor has
clustered chromatic number at most $496t$.
\begin{THM}[Kawarabayashi~\cite{Kawarabayashi2008}]\label{thm:weakoddhad}
For each integer $t \geq 2$, there is an integer $C = C(t)$ such that every graph $G$ with no odd $K_t$ minor is $496t$-colorable with clustering $C$.
\end{THM}

We  improve $496t$ to $10t-13$ as follows.

\begin{THM}\label{mainthm2}
For each integer $t \geq 2$, there exists an integer $C = C(t)$ such that every graph $G$ with no odd $K_t$ minor is $(10t-13)$-colorable with clustering $C$.
\end{THM}

We also remark that $10t-13$ cannot be reduced to the number less than $t-1$.
Both Theorems~\ref{mainthm} and~\ref{mainthm2} cannot be extended for list-colorings, which we will discuss in Section~\ref{sec:conclude}.

The paper is organized as follows.
In Section~\ref{sec:survey} we review related results on minors,
some of which will be used in our proof.
We briefly introduce some basic notions in Section~\ref{sec:term}, discuss the structure of graphs with no odd $K_t$ minor in Section~\ref{sec:structure}, and prove Theorems~\ref{mainthm} and~\ref{mainthm2} in Section~\ref{sec:coloring}. In Section~\ref{sec:conclude}, we make some further remarks, including extension of our main results to a slightly larger class of graphs. An appendix reviews elementary concepts of signed graphs and minors.

\section{Previous results on improper coloring
  and forbidden minors}\label{sec:survey}
There are many studies regarding improper colorings of graphs with forbidden minors. Kawarabayashi and Mohar~\cite{KM2007a} proved that
the clustered chromatic number of the class of graphs with no $K_t$ minor is at most $\lceil \frac{31}{2}t \rceil$.
This was improved to $\lceil \frac{7t-3}{2} \rceil$ by Wood~\cite{Wood2010}. Edwards, Kang, Kim, Oum, and Seymour~\cite{EKKOS2014} investigated defective coloring of graphs with no $K_t$ minor, and proved that
the defective chromatic number of the graphs with no $K_t$ minor
equals $t-1$.

\begin{THM}[Edwards, Kang, Kim, Oum, and Seymour~\cite{EKKOS2014}]\label{thm:weakhad}
For each integer $t \geq 1$, there exists an integer $s(t) = O(t^2 \log t)$ such that every graph $G$ with no $K_t$ minor is $(t-1)$-colorable with defect $s(t)$. Moreover, this is sharp in the sense that we cannot reduce the number $t-1$ of sets to $t-2$.
\end{THM}

They also proved that the clustered chromatic number of the class of graphs with no $K_t$ minor is at most $4t-4$. Liu and Oum~\cite{liu2017partitioning} proved that for every graph $H$, every graph $G$ with no $H$-minor and maximum degree at most $\Delta$ is $3$-colorable with clustering $f(H,\Delta)$ for some function $f$, which generalizes the result of Esperet and Joret~\cite{esperet2014} for graphs embeddable on surfaces of bounded Euler genus. Combined with Theorem~\ref{thm:weakhad}, this implies that the clustered chromatic number of the class of graphs with no $K_t$ minor is at most $3t-3$. Van den Heuvel and Wood~\cite{van1704improper} proved that every graph with no $K_t$ minor is $(2t-2)$-colorable with clustering $\lceil \frac{t-2}{2} \rceil$, using different proofs that do not rely on the excluded minor structure theorem. Dvo\v{r}\'ak and Norin~\cite{dvovrak2017islands} proved that the clustered chromatic number of the class of graphs with no $K_t$ minor and treewidth at most $w$ is at most $t-1$, and announced that the clustered chromatic number of the class of graphs with no $K_t$ minor equals $t-1$.

What happens if the forbidden graph $H$ is not complete?
Let $I_t$ be a graph on $t$ vertices with no edges, and for graphs $G$ and $H$, let $G+H$ be a graph obtained from the disjoint union of $G$ and $H$ by adding an edge between each vertex of $G$ and each vertex of $H$. For positive integers $s$ and $t$, let $K_{s,t}^*$ be a graph obtained from $K_s + I_t$ by subdividing every edge joining vertices of the subgraph $K_s$ once. Recently, Ossona de Mendez, Oum, and Wood~\cite{ossonademendez2016} investigated defective coloring for various graph classes. One of their results implies the following, which extends Theorem~\ref{thm:weakhad} to a larger class of graphs.

\begin{THM}[Ossona de Mendez, Oum, and Wood~\cite{ossonademendez2016}]\label{thm:genweakhad}
  For integers $s,t \geq 1$ and real numbers $\delta_1, \delta_2 > 0$, there exists $M = M(s,t,\delta_1 , \delta_2)$ such that
  every graph $G$ satisfying the following three conditions is $s$-colorable with defect $M$.
\begin{enumerate}
\item $G$ contains no $K_{s,t}^*$ as a subgraph.
\item Every subgraph of $G$ has average degree at most $\delta_1$.
\item For every graph $H$ whose $1$-subdivision is a subgraph of $G$, the average degree of $H$ is at most $\delta_2$.
\end{enumerate}
Moreover, this is sharp in the sense that we cannot reduce the number $s$ to $s-1$.
\end{THM}

Since $K_{s,t}^*$ is a bipartite $K_s + I_t$ subdivision, Theorem~\ref{thm:genweakhad} implies that the defective chromatic number of the class of graphs with no bipartite $K_s + I_t$ subdivision equals $s$ as follows. (The lower bound is obtained
by Theorem~\ref{thm:weakhad}.)

\begin{COR}\label{cor:genbipweakhad}
For positive integers $s$ and $t$, there is an integer $N = N(s,t)$ such that every graph $G$ with no bipartite $K_s + I_t$ subdivision is $s$-colorable with defect $N$.
\end{COR}

\begin{proof}
  By~\cite{BT98, KS96}, there is $c_0 > 0$ such that for each integer $p \geq 1$, every $n$-vertex graph with average degree at least $c_0 p^2$ contains $K_p$ as a topological minor. Since $G$ contains no bipartite $K_s + I_t$ subdivision, the graph $G$ contains no $K_{s,t}^*$ as a subgraph, and no bipartite $K_{s+t}$ subdivision. If there is a subgraph $H$ of $G$ with average degree at least $2c_0 (s+t)^2$, then let $H_0$ be a bipartite spanning subgraph of $H$ with at least $|E(H)|/2$ edges. Since $H_0$ has average degree at least $c_0 (s+t)^2$, it follows that $H$ contains a bipartite $K_{s+t}$ subdivision, contradicting the assumption on $G$. Hence every subgraph of $G$ has average degree at most $2c_0 (s+t)^2$. If there is a graph $H$ whose $1$-subdivision is a subgraph of $G$, then the average degree of $H$ is at most $c_0 (s+t)^2$, because otherwise $H$ contains $K_{s+t}$ as a topological minor, and
  its $1$-subdivision is bipartite.

  By Theorem~\ref{thm:genweakhad},
  $G$ is $s$-colorable with defect $M(s,t,2c_0(s+t)^2, c_0(s+t)^2)$.
\end{proof}

Since $K_t + I_1$ is isomorphic to $K_{t+1}$, Theorem~\ref{thm:weakhad} can be extended to graphs with no bipartite clique subdivision.

\begin{COR}
  For each integer $t \geq 1$,
  the defective chromatic number
  of the class of graphs with no bipartite $K_{t+1}$ subdivision
  equals $t$.

\end{COR}

Mohar, Reed, and Wood~\cite{mohar2017colourings} studied clustered colorings of graphs with no $C_{k+1}$ minor, where $C_{k+1}$ denotes a cycle of length $k+1$. They proved that for every integer $k \geq 2$, every graph with no $C_{k+1}$ minor is $\lfloor 3 \log_2 k \rfloor$-colorable with clustering $k$, and the number of colors is asymptotically tight. Norin, Scott, Seymour, and Wood~\cite{norin2017clustered} proved that for every graph $H$,
the clustered chromatic number of the class of $H$-minor-free graphs is tied to the tree-depth of $H$, giving a partial answer to a conjecture in~\cite{ossonademendez2016}.

Liu and Oum~\cite{liu2017partitioning} proved that for every graph $H$ and every integer $\Delta$,
the class of graphs with no odd $H$ minor and maximum degree at most $\Delta$ has the clustered chromatic number at most $3$.

\begin{THM}[Liu and Oum~\cite{liu2017partitioning}]\label{partition}
  For every graph $H$ and every integer $\Delta\ge0$,
  there is $C = C(H, \Delta)$ such that for every graph $G$ with maximum degree at most $\Delta$ and no odd $H$ minor, there are
  pairwise disjoint subsets $V_1 , V_2 , V_3$ of $V(G)$ such that $V_1 \cup V_2 \cup V_3 = V(G)$ and every component of $G[V_i]$ has at most $C$ vertices for $i=1,2,3$.
\end{THM}

\section{Preliminaries}\label{sec:term}
We follow the definitions in~\cite{Diestel2010} unless stated otherwise. In this section, $G$ and $H$ always denote graphs.
For each integer $N \geq 0$, let \emph{$[N]$} be the set $\left \{1 , \dots , N \right \}$. If $N=0$, then $[N]$ is an empty set.
For a graph $H$, let $\Delta(H)$ be the maximum degree of vertices in $H$. For $S \subseteq V(H)$, let $H[S]$ be the subgraph of $H$ induced by $S$.

A subset $S \subseteq V(G)$ is \emph{stable} if no two vertices in $S$ are adjacent. Let $G \cup H$ and $G \cap H$ be graphs $(V(G)\cup V(H), E(G) \cup E(H))$ and $(V(G)\cap V(H), E(G) \cap E(H))$, respectively. A pair $(A,B)$ of subgraphs of $G$ is a \emph{separation} of $G$ if $G = A \cup B$. The \emph{order} of a separation $(A,B)$ of $G$ is $|V(A \cap B)|$. A \emph{bipartition} $\left \{X,Y \right \}$ of a bipartite graph $H$ is a set of two disjoint stable subsets such that $X \cup Y = V(G)$.

\subsection*{Coloring} For a subset $S$ of vertices of $G$ and a set $T$ of \emph{colors}, a function $\alpha : S \to T$ is a \emph{coloring} on $S$. A \emph{color class} of a coloring $\alpha : S \to T$ is $\alpha^{-1}(i)$ for some $i\in T$. A subgraph $H$ of $G$ is \emph{monochromatic} if every vertex of $H$ has the same $\alpha$ value.

A coloring $\alpha$ on $S$ is \emph{proper} if $\alpha(u) \ne \alpha(v)$ for every $uv \in E(G[S])$; equivalently, every color class of $\alpha$ is stable. For a nonnegative integer $k$, a graph $G$ is \emph{$k$-colorable} if there is a proper coloring $\alpha : V(G) \to [k]$. For integers $k,d \geq 0$, a graph $G$ is \emph{$k$-colorable} with \emph{defect $d$} if there is a coloring $\alpha : V(G) \to [k]$ such that every color class induces a graph with maximum degree at most $d$, and $G$ is \emph{$k$-colorable} with \emph{clustering $d$} if there is a coloring $\alpha : V(G) \to [k]$ such that every color class induces a subgraph with no connected components having more than $d$ vertices.

\subsection*{Paths}
A path is a graph that consists of $k$ vertices $v_1 , \dots , v_k$ for some integer $k \geq 1$ and $k-1$ edges $v_1v_2 , \dots , v_{k-1}v_k$. The vertices $v_1$ and $v_k$ are called \emph{ends}, and all other vertices $v_2 , \dots , v_{k-1}$ are called \emph{internal} vertices.

The path $P$ \emph{joins} $u,v \in V(G)$ if $u$ and $v$ are ends of $P$. For vertices $v,w \in V(P)$, $P(v,w)$ denotes the subpath of $P$ with the ends $v$ and $w$. 

A path $P$ \emph{joins} two sets $V_1 , V_2 \subseteq V(G)$ if it joins a vertex in $V_1$ and a vertex in $V_2$. The \emph{length} of a path is its number of edges. The \emph{parity} of a path $P$ is the parity of its length.

Two paths $P$ and $Q$ in $G$ are \emph{vertex-disjoint} if $V(P) \cap V(Q) = \emptyset$. They are \emph{internally disjoint}\footnote{It is called \emph{independent} in~\cite{Diestel2010}.} if
every common vertex of $P$ and $Q$ is an end of both $P$ and $Q$.

For $S \subseteq V(G)$, an \emph{$S$-path} is a path in $G$ that joins two \emph{distinct} vertices in $S$. For a coloring $\alpha : S \to \left \{1,2 \right \}$, an $S$-path $P$
from $u$ to $v$ in $G$ is \emph{parity-breaking with respect to $\alpha$} if \[|E(P)| \not\equiv \alpha(u) - \alpha(v) \pmod 2.\] For a connected bipartite subgraph $H$ of $G$ with a proper coloring $\beta : V(H) \to \left \{1,2 \right \}$, a $V(H)$-path $P$ in $G$ is \emph{parity-breaking with respect to $H$} if $P$ is parity-breaking with respect to $\beta$. This is well defined since a proper coloring of $H$ is unique up to permuting colors. We will use the following observation in Section~\ref{sec:structure}.

\begin{OBS}\label{pasting}
\mbox{}
\begin{enumerate}
\item For $S\subseteq V(G)$ and a coloring $\alpha:S\to\{1,2\}$, let $P$, $Q$ be internally disjoint $S$-paths sharing precisely one end. Then the $S$-path $P\cup Q$ is parity-breaking with respect to $\alpha$ if and only if exactly one of $P$ and $Q$ is parity-breaking with respect to $\alpha$.

\item For a connected bipartite subgraph $H$ of $G$, no path in $H$ is parity-breaking with respect to $H$.
\end{enumerate}
\end{OBS}

\subsection*{Minors} A graph $H$ is a \emph{minor} of $G$ if a graph isomorphic to $H$ can be obtained from $G$ by deleting vertices or edges and contracting edges. If there are edges $wu$ and $wv$ for some vertex $w \notin \left \{u,v \right \}$ and we contract an edge $uv$, then one of these two edges is removed after contraction to avoid parallel edges. A graph $G$ \emph{contains an $H$-minor} (or $H$ as a \emph{minor}) if $H$ is a minor of $G$.

\subsection*{Topological minors} An \emph{$H$-subdivision} is a graph obtained from $H$ by subdividing edges, where edges may be subdivided more than once. A graph $G$ \emph{contains an $H$-subdivision} (or $H$ as a \emph{topological minor}), if $G$ contains a subgraph isomorphic to an  $H$-subdivision.

Since every $H$-subdivision $H'$ is built from $H$ by replacing all edges of $H$ with internally disjoint paths called the \emph{linking paths}, there are vertices of $H'$ that correspond to vertices of $H$, which we call \emph{branch vertices}.

\subsection*{Odd minors} For $S \subseteq V(G)$ and a coloring $\alpha : S \to \left \{1,2 \right \}$, an edge $uv \in E(G[S])$ is \emph{bichromatic} if $\alpha(u) \ne \alpha(v)$, and is \emph{monochromatic} otherwise.

For graphs $G$ and $H$, $G$ contains $H$ as an \emph{odd minor} if there exist vertex-disjoint subgraphs $\left \{T_u \right \}_{u \in V(H)}$ in $G$ which are trees, and a coloring $\alpha : \bigcup_{u \in V(H)}{V(T_u)} \to \left \{1,2 \right \}$ such that for every $u \in V(H)$, every edge in $T_u$ is bichromatic, and for every edge $vw \in E(H)$, there is a monochromatic edge $e \in E(G)$ that joins $V(T_v)$ and $V(T_w)$.

We will use the following alternative definition in Section~\ref{sec:structure}.

\begin{OBS}\label{cond}
For graphs $G$ and $H$, $G$ contains $H$ as an odd minor if and only if there exist vertex-disjoint subgraphs $\left \{T_u \right \}_{u \in V(H)}$ in $G$ which are trees, and a coloring $\alpha : \bigcup_{u \in V(H)}{V(T_u)} \to \left \{1,2 \right \}$ such that 

\begin{enumerate}
\item for every $u \in V(H)$, every edge in $T_u$ is bichromatic,

\item there are internally disjoint paths $\left \{P_{e} \right \}_{e \in E(H)}$ in $G$,

\item for every $e = vw \in E(H)$, $P_{e}$ joins $V(T_v)$ and $V(T_w)$, has no internal vertex in $\bigcup_{u \in V(H)}{V(T_u)}$, and is parity-breaking with respect to $\alpha$.
\end{enumerate}
\end{OBS}

We remark that
for two graphs $G$ and $H$,
$G$ contains $H$ as an odd minor if and only if a signed graph $(G,E(G))$ contains a signed graph $(H , E(H))$ as a minor, which we will discuss in Appendix~\ref{app}.

\section{The structure of graphs with no odd clique minor}\label{sec:structure}
The proof of Theorem~\ref{thm:weakhad} is based on the fact that every graph with no $K_t$ minor has a vertex of degree at most $c_t$ for some $c_t$. In contrast to graphs with no $K_t$ minor, graphs with no odd $K_t$ minor may have arbitrarily large minimum degree; for example, complete bipartite graphs have no odd $K_3$ minor.

To prove Theorems~\ref{mainthm} and~\ref{mainthm2}, we use the following strategy similar to the one by Geelen, Gerards, Reed, Seymour, and Vetta~\cite{GGRSV2009}. If a graph $G$ has no bipartite subdivision of some graph, then we apply Corollary~\ref{cor:genbipweakhad}. Otherwise, we will show in Theorem~\ref{decomp} that $G$ contains a bipartite block after removing few vertices, which allows us to use precoloring arguments in the following section.

First of all, we describe how to find an odd $K_t$ minor in a graph $G$ if the graph contains a bipartite $K_{2t-2} + I_t$ subdivision with many vertex-disjoint parity-breaking paths between branch vertices of the subdivision.

\begin{LEM}\label{force}
For $t \geq 2$, let $G$ be a graph that contains a bipartite $K_{2t-2} + I_{t}$ subdivision~$H$, and $C$ be the set of all branch vertices of $K_{2t-2} + I_{t}$ in $H$. If there are $t-1$ vertex-disjoint parity-breaking $C$-paths with respect to $H$, then $G$ contains an odd $K_t$ minor.
\end{LEM}

\begin{proof}
For convenience, we identify each vertex in $C$ with its corresponding vertex of $K_{2t-2} + I_t$.

For a collection $\mathcal Q$ of paths, let $\ell(\mathcal{Q}) = \sum_{P \in \mathcal{Q}}{|E(P)|}$. Let $\mathcal{P}$ be a collection of $t-1$ vertex-disjoint parity-breaking $C$-paths with respect to $H$, satisfying the following.

\begin{enumerate}[(a)]
\item $\sum_{P \in \mathcal{P}}{|E(P) \setminus E(H)|}$ is minimum, and
\item  subject to (a), $\ell (\mathcal P)$ is minimum.
\end{enumerate}

Note that every vertex in $C$ is not an internal vertex of a path in $\mathcal{P}$. To see this, if a vertex in $C$ is an internal vertex of a path $Q \in \mathcal{P}$, then $Q$ contains a proper subpath $Q'$ that is a parity-breaking $C$-path with respect to $H$. For $\mathcal{Q} := (\mathcal{P} \setminus \left \{Q \right \}) \cup \left \{Q' \right \}$ it follows that $\sum_{P \in \mathcal{Q}}{|E(P) \setminus E(H)|} \leq \sum_{P \in \mathcal{P}}{|E(P) \setminus E(H)|}$ and $\ell(\mathcal{Q}) < \ell(\mathcal{P})$, contradicting our choice of $\mathcal{P}$.

For distinct $u,v \in C$, if $uv \in E(K_{2t-2} + I_t)$, then let $Q_{u,v}$ be the linking path from $u$ to $v$ in $H$. If $uv \notin E(K_{2t-2} + I_t)$, then let $Q_{u,v}$ be a graph with two vertices $u$ and $v$ and no edges. Note that there are $2t-2$ branch vertices that appear in paths in $\mathcal{P}$, and $t$ unused branch vertices. Let $C_0 \subseteq C$ be the set of those unused $t$ branch vertices.

\begin{CLAIM}\label{path} Let $u$, $v$ be distinct vertices in $C$.
\begin{enumerate}
\item If $u, v \in C_0$, then no path $P \in \mathcal{P}$ intersects $Q_{u,v}$.%
\footnote{Two subgraphs \emph{intersect}  if they share at least one common vertex.}
\item If $u \in C_0$, then $Q_{u,v}$ intersects at most one path $P$ in $\mathcal{P}$, and if so, then the intersection is a subpath of both $Q_{u,v}$ and $P$ and contains $v$.
\end{enumerate}
\end{CLAIM}

\begin{proof}[Proof of Claim~\ref{path}]
\renewcommand{\qedsymbol}{\qedblack}
We may assume $uv \in E(K_{2t-2} + I_t)$, otherwise $V(Q_{u,v}) \subseteq C$ and the claim is trivial. Let $u \in C_0$, and suppose $Q_{u,v}$ intersects a path in $\mathcal{P}$. Since no path in $\mathcal{P}$ intersects $u$, starting from $u$ and following $Q_{u,v}$ we arrive at the first vertex $w \in V(Q_{u,v})$ on some path $P \in \mathcal{P}$.

Write $P = A \cup B$, where $A$ and $B$ are two subpaths of $P$ with the only common vertex $w$. Since $w \in V(H)$, either $A$ or $B$ is parity-breaking with respect to $H$, and we may assume $A$ is parity-breaking with respect to $H$. By Observation~\ref{pasting}, a path $R = A \cup Q_{u,v}(w,u)$ is parity-breaking with respect to $H$, and it intersects no path in $\mathcal{P}$ other than $P$. Therefore, we conclude that $(\mathcal{P} \setminus \left \{P \right \} ) \cup \left \{R \right \}$ is a set of $t-1$ vertex-disjoint parity-breaking $C$-paths with respect to $H$. By our assumption on $\mathcal{P}$, $P$ does not have more edges not in $H$ than $R$. This implies $E(B) \subseteq E(H)$, and thus $B = Q_{u,v}(w , v)$ since one of the ends of $B$ is in $C$ and no path in $\mathcal{P}$ intersects $u$. Note that $A$ intersects $Q_{u,v}$ only at $w$, since $B = Q_{u,v}(w,v)$.

Since $P$ and $Q_{u,v}$ share a common subpath from $w$ to $v$ and $w$ is the only vertex that belongs to both $Q_{u,v}(w,u)$ and some path in $\mathcal{P}$, $P$ is the only path that intersects $Q_{u,v}$. In particular, $B = P \cap Q_{u,v}$ is a path that contains $v$, and $v \notin C_0$.
\end{proof}

Let $\mathcal{P} = \left \{P_1 , \dots , P_{t-1} \right \}$, and for $1 \leq i \leq t-1$, let $x_i$ and $y_i$ be the ends of $P_i$. Let $C_1 = \left \{x_1 , \dots , x_{t-1} \right \}$ and $C_2 = \left \{y_1 , \dots , y_{t-1} \right \}$.

For $v \in C$, if $v$ corresponds to a vertex in the subgraph $K_{2t-2}$ of $K_{2t-2} + I_{t}$, then we call $v$ \emph{Type-A}. Otherwise we call $v$ \emph{Type-B}. Let $q$ be the number of $i$'s ($1 \leq i \leq t-1$) such that both $x_i$ and $y_i$ are Type-A, and $r$ be the number of $i$'s such that exactly one of $x_i$ and $y_i$ is Type-A, and $s = t-1 - q - r$. Then there are $(2t-2) - (2q+r)$ vertices of Type-A in $C_0$. Since $q+r+s = t-1$, it follows that $(2t-2) - (2q+r) \geq r+s$. Therefore, the number of vertices in $C_0$ of Type-A is at least $r+s$. Thus, we choose an ordering  $z_1 , \dots , z_t$ of the vertices in $C_0$ such that for $1 \leq i \leq t-1$ if $x_i$ or $y_i$ is Type-B, then $z_i$ is a vertex of Type-A.

In summary, $z_i x_i , z_i y_i \in E(K_{2t-2} + I_t)$ for $1 \leq i \leq t-1$. Equivalently, if $z_i$ is Type-B, then both $x_i$ and $y_i$ are Type-A.

Let $\beta : V(H) \to \left \{1,2 \right \}$ be a proper coloring of $H$ unique up to permuting colors. In order to find an odd $K_t$ minor, we now aim to construct vertex-disjoint subgraphs $M_1 , \dots , M_t$ in $G$ which are trees and a coloring $\alpha : \bigcup_{i=1}^{t}{V(M_i)} \to \left \{1,2 \right \}$ as in Observation~\ref{cond}. 

For $1 \leq i \leq t-1$, let $M_i = P_i \cup Q_{z_i , y_i}$ if $z_i$ is Type-A, and $M_i = P_i \cup Q_{z_i , x_i}$ if $z_i$ is Type-B. Let $M_t$ be a graph with the only vertex $z_t$. 
If $i<t$, then by Claim~\ref{path}, $P_i\cap Q_{z_i,y_i}$ or $P_i\cap Q_{z_i,x_i}$ is a subpath of $P_i$ and so $M_i$ is a tree with maximum degree at most $3$ and at most one vertex of degree~$3$.
We choose a coloring $\alpha:\bigcup_{i=1}^t V(M_i)\to \{1,2\}$ such that 
\begin{enumerate}
\item $\alpha$ on $V(M_i)$ is a proper coloring of $M_i$ and $\alpha(x_i)=\beta(x_i)$ for all $1\le i \le t-1$,  and
\item $\alpha(z_t) = \beta(z_t)$ if and only if $z_t$ is Type-B.
\end{enumerate}
 Observation~\ref{pasting} implies that, for $1\le i\le t-1$, $\alpha(y_i) \ne \beta(y_i)$ since $P_i$ is parity-breaking with respect to $H$ and $\alpha(z_i) = \beta(z_i)$ if and only if $z_i$ is Type-B. 

For $1 \leq i < j \leq t$, we are now ready to construct a path $P_{i,j}$ joining $V(M_i)$ and $V(M_j)$ that is parity-breaking with respect to $\alpha$ and satisfies Observation~\ref{cond}. This will show that $G$ contains an odd $K_t$ minor. The structure of $P_{i,j}$ depends on the types of $z_i$ and $z_j$.
\smallskip

\noindent
\textbf{Case 1}. Both $z_i$ and $z_j$ are Type-A.

Following $Q_{z_j , x_i}$ from $z_j$ to $x_i$, we arrive at the first vertex $a_{i,j}$ in $V(P_i) \cap V(Q_{z_j , x_i})$. By Claim~\ref{path}, $Q_{z_j , x_i}$ and $P_i$ share the subpath from $a_{i,j}$ to $x_i$. Since $P_i (x_i , a_{i,j})$ is in $H$ and $\alpha(x_i) = \beta(x_i)$, it follows that $\alpha(a_{i,j}) = \beta(a_{i,j})$ by Observation~\ref{pasting}. Let us define $P_{i,j} = Q_{z_j , x_i} (z_j , a_{i,j})$. Since $\alpha(z_j) \ne \beta(z_j)$, $\alpha(a_{i,j}) = \beta(a_{i,j})$ and $P_{i,j}$ is a subpath of $Q_{z_j , x_i}$, we conclude that $P_{i,j}$ is parity-breaking with respect to $\alpha$ by Observation~\ref{pasting}.
\smallskip

\noindent
\textbf{Case 2}. $z_i$ and $z_j$ are of different types.

Let us define $P_{i,j} := Q_{z_i , z_j}$. By Claim~\ref{path}, $P_{i,j}$ intersects no path in $\mathcal{P}$. Since $\alpha(z_i) = \beta(z_i)$, $\alpha(z_j) \ne \beta(z_j)$, and $Q_{z_i , z_j}$ is in $H$, $P_{i,j}$ is parity-breaking with respect to $\alpha$ by Observation~\ref{pasting}.\smallskip

\noindent
\textbf{Case 3}. Both $z_i$ and $z_j$ are Type-B.

Since $z_i$ is Type-B, $y_i$ is Type-A. Following $Q_{z_j , y_i}$ from $z_j$ to $y_i$, we arrive at the first vertex $a_{i,j}$ in $V(P_i) \cap V(Q_{z_j , y_i})$. Claim~\ref{path} implies that $Q_{z_j , y_i}$ and $P_i$ share the subpath from $a_{i,j}$ to $y_i$. Since $P_i (y_i , a_{i,j})$ is in $H$ and $\alpha(y_i) \ne \beta(y_i)$, it follows that $\alpha(a_{i,j}) \ne \beta(a_{i,j})$. Let us define $P_{i,j} := Q_{z_j , y_i} (z_j , a_{i,j})$. Since $\alpha(z_j) = \beta(z_j)$, $\alpha(a_{i,j}) \ne \beta(a_{i,j})$ and $P_{i,j}$ is in $H$, $P_{i,j}$ is parity-breaking with respect to $\alpha$ by Observation~\ref{pasting}.
\end{proof}

We use the following lemma, which asserts that the family of $S$-paths of odd length satisfies the Erd\H{o}s-P\'{o}sa property.

\begin{LEM}[Geelen, Gerards, Reed, Seymour, and Vetta~{\cite[Lemma 11]{GGRSV2009}}]\label{oddpath}
Let $G$ be a graph and $S \subseteq V(G)$. For every integer $\ell \geq 1$, $G$ contains $\ell$ vertex-disjoint $S$-paths of odd length, or there is $X \subseteq V(G)$ with $|X| \leq 2 \ell - 2$ such that $G \setminus X$ contains no $S$-path of odd length.
\end{LEM}

\begin{OBS}\label{smallsubdiv}
Let $G$ and $H$ be graphs and $X \subseteq V(G)$. If $G$ contains an $H$-subdivision $K$, then $G \setminus X$ contains an $H'$-subdivision $K'$ such that $H' = H \setminus Y$ for some $Y \subseteq V(H)$ with $|Y| \leq |X|$ and $K'$ is a subgraph of $K$.
\end{OBS}
\begin{proof}
It is easy to see that if $G$ has an $H$-subdivision $K$ and $v$ is a vertex of $K$,
then there is a vertex $w$ of $H$ such that $G\setminus v$ has a $(H\setminus w)$-subdivision.
\end{proof}

The following lemma is a variation of~\cite[Lemma 15]{GGRSV2009}.

\begin{LEM}\label{paritybreak}
Let $\ell$ be a positive integer and $G$ be a graph. Let $H$ be a bipartite $K_s + I_t$ subdivision in $G$ for integers $s \geq 2 \ell$ and $t \geq 1$, and $C$ be the set of all branch vertices in $H$. At least one of the following holds.
\begin{itemize}
\item
There exists $X \subseteq V(G)$ with $|X| \leq 2\ell-2$ such that $G-X$ has a bipartite block $U$ that contains at least $s + t - |X|$ vertices in $C \setminus X$ and all linking paths in $H$ between them.
\item
$G$ has $\ell$ vertex-disjoint parity-breaking $C$-paths with respect to $H$.
\end{itemize}
\end{LEM}
\begin{proof}
(1) \emph{We claim that either there are $\ell$ vertex-disjoint parity-breaking $C$-paths in $G$ with respect to $H$, or there is $X \subseteq V(G)$ with $|X| \leq 2\ell - 2$ such that $G \setminus X$ contains no parity-breaking $C$-path with respect to $H$.}

Let $\left \{L , R \right \}$ be the unique bipartition of $H$. 
Without loss of generality, we may assume that every linking path corresponding to an edge in $K_s + I_t$ has even length, because otherwise, for every branch vertex $v \in C \cap L$, we subdivide each edge $e \in E(G)$ incident with $v$ once. This gives an $H$-subdivision $H'$ and a $G$-subdivision $G'$ such that $H'$ is a bipartite subgraph of $G'$. We may assume $V(H) \subseteq V(H')$ and $V(G) \subseteq V(G')$, and then  every vertex in $V(G') \setminus V(G)$ has degree 2. Note that all vertices in $C$ are in the same part of the bipartition of $H'$, and thus every path in $H'$ between vertices in $C$ has even length. It is easy to check the following.
\begin{itemize}
\item A $C$-path of odd length in $G'$ corresponds to a parity-breaking $C$-path in $G$ with respect to~$H$.
\item If there is $X' \subseteq V(G')$ with $|X'| \leq 2 \ell - 2$ such that $G' \setminus X'$ contains no $C$-path of odd length, then we may assume $X' \subseteq V(G)$ since every vertex in $V(G') \setminus V(G)$ has degree 2. 
\end{itemize}

Lemma~\ref{oddpath} claims that either $G'$ contains $\ell$ vertex-disjoint $C$-paths of odd length, or there is $X' \subseteq V(G')$ with $|X'| \leq 2\ell-2$ such that $G' \setminus X'$ contains no $C$-path of odd length. This proves (1).
\medskip

Suppose $G$ contains no $\ell$ vertex-disjoint parity-breaking $C$-paths with respect to $H$. By (1), there is $X \subseteq V(G)$ with $|X| \leq 2\ell - 2$ such that $G \setminus X$ contains no parity-breaking $C$-path with respect to $H$. For convenience, we identify each vertex in $C$ with  its corresponding vertex of $K_s + I_t$. For distinct $u,v \in C$, if $uv \in E(K_s + I_t)$ then let $Q_{u,v}$  be the linking path from $u$ to $v$ in $H$. If $uv \notin E(K_s + I_t)$, then let $Q_{u,v}$ be a graph with two vertices $u$ and $v$ and no edges.
\medskip

(2) \emph{We claim that there is a block $U$ in $G \setminus X$ containing at least $s + t - |X|$ vertices in $C \setminus X$ and all linking paths in $H$ between them.}

By Observation~\ref{smallsubdiv}, $G \setminus X$ contains a $(K_a + I_b)$-subdivision $K$ such that $K_a + I_b = (K_s + I_t) \setminus Y$ for some $Y \subseteq V(K_s + I_t)$ with $|Y| \leq |X|$ and $K$ is a subgraph of $H$. Let $T$ be the set of all branch vertices in $K$, where $|T| \geq a+b = s+t - |Y| \geq s+t - |X|$.

Since $a \geq s - |Y| \geq 2$ and $a + b = s+t - |Y| \geq 3$, $K_a + I_b$ is 2-connected. Therefore, $K$ is 2-connected and all vertices in $T$ are in the same block of $G \setminus X$.
\medskip

(3) \emph{We claim that $U$ is bipartite.}

Suppose $U$ contains an odd-length cycle $D$. For two distinct vertices $u,v \in C \cap V(U)$, there are two vertex-disjoint paths in $U$ joining $\left \{u,v \right \}$ and $V(D)$ by Menger's theorem. Using these paths, we obtain both an odd-length path and an even-length path from $u$ to $v$ in $U$. One of those paths is a parity-breaking $C$-path with respect to $H$, contradicting that $G \setminus X$ has no parity-breaking $C$-path with respect to $H$.
\end{proof}

Now we are ready to prove the main theorem of this section.

\begin{THM}\label{decomp}
Let $t \geq 2$ be an integer, and $G$ be a graph. If $G$ contains no odd $K_t$ minor and contains a bipartite $K_{2t-2} + I_{t}$ subdivision, then there is $X \subseteq V(G)$ with $|X| \leq 2t-4$ such that $G \setminus X$ contains a bipartite block $U$ having at least $t+3$ vertices.
\end{THM}

\begin{proof}
Let $H$ be a bipartite $K_{2t-2} + I_t$ subdivision of $G$, and $C = \left \{v_1 , \dots , v_{3t-2} \right \}$ be the set of all branch vertices in $H$. For convenience, we identify each vertex in $C$ with its corresponding vertex in $V(K_{2t-2} + I_t)$. Let $C_1 \subseteq C$ be the set of branch vertices corresponding to vertices in $K_{2t-2}$, and $C_2 = C \setminus C_1$ be the set of branch vertices corresponding to vertices in $I_t$.

By Lemmas~\ref{force} and~\ref{paritybreak}, there exists $X \subseteq V(G)$ with $|X| \leq 2t-4$ such that $G\setminus X$ has a bipartite block $U$ containing at least $(3t-2) - |X|$ vertices in $C \setminus X$ and all linking paths between them. Let $C' \subseteq C$ be those $(3t-2) - |X| \geq t+2$ branch vertices in $U$, and $H'$ be the union of all linking paths between vertices in $C'$, which is a subgraph of $U$. 

Recall that we identified $C \subseteq V(H)$ with $V(K_{2t-2} + I_t)$. Since vertices in $C_1$ form a clique of $K_{2t-2} + I_t$ and $|C' \cap C_1| \geq |C'| - |C_2| \geq 2$, the subgraph of $K_{2t-2} + I_t$ induced by $C'$ is not bipartite. To obtain $H'$ from the induced subgraph of $K_{2t-2} + I_t$, we should subdivide edges at least once, because $H'$ is bipartite. Thus $H'$ contains a vertex other than vertices in $C'$, implying $|V(U)| \geq |V(H')| \geq |C'|+1 \geq t+3$.
\end{proof}

\section{Proofs of Theorems~\ref{mainthm} and~\ref{mainthm2}}\label{sec:coloring}
For a class $\mathcal{F}$ of graphs and an integer $d \geq 0$, a graph $G$ has a \emph{$(d, \mathcal{F})$-coloring} if there is $f : V(G) \to [d]$ such that $G[f^{-1}(\{i\})]$ is in $\mathcal{F}$ for all $i \in [d]$. A class $\mathcal{F}$ of graphs is \emph{closed under isomorphisms} if for all $G \in \mathcal{F}$, every graph isomorphic to $G$ is in $\mathcal{F}$. A class $\mathcal{F}$ of graphs is \emph{closed under taking disjoint unions} if for all $G,H \in \mathcal{F}$, the disjoint union of $G$ and $H$ is in $\mathcal{F}$.

Now we are ready to prove the following lemma,
following the idea of Kawarabayashi and Mohar~\cite{KM2007a}.

\begin{LEM}\label{precoloring}
Let $t \geq 2$ and $d \geq 3$ be integers and $\mathcal{F}$ be a class of graphs closed under isomorphisms and taking disjoint unions, which satisfies the following.
\begin{itemize}
\item[(i)] $\mathcal{F}$ contains every graph with at most $4t-7$ vertices.
\item[(ii)] If a graph $H$ contains no odd $K_t$ minor and no bipartite $K_{2t-2} + I_{t}$ subdivision, then $H$ has a $(d , \mathcal{F})$-coloring. 
\end{itemize}

Then every graph with no odd $K_t$ minor has a $(d+4t-7, \mathcal{F})$-coloring.
\end{LEM}

\begin{proof}
We prove the following stronger claim.

\medskip
\noindent\textbf{Claim}. \emph{Let $G$ be a graph with no odd $K_t$ minor, $Z \subseteq V(G)$ with $|Z| \leq 4t-7$, and $f : Z \to [d+4t-7]$ be a coloring. Then $G$ has a $(d+4t-7, \mathcal{F})$-coloring $g$ that satisfies the following.
\begin{enumerate}[(a)]
\item For every $z \in Z$, $f(z) = g(z)$.
\item For every $v \in Z$ and its neighbor $w \notin Z$, $g(v) \ne g(w)$.
\end{enumerate}}

Let $G$ be a counterexample with the minimum $|V(G)|+|E(G)|$. As the claim is true for graphs with at most $4t-7$ vertices by giving distinct colors to each vertex not in $Z$, $|V(G)| \geq 4t-6$.
\medskip

(1) \emph{$Z$ is stable.}

Suppose there are adjacent $z_1 , z_2 \in Z$. Applying the claim on $G' = G \setminus z_1 z_2$ with the same $Z$ and $f$, $G'$ has a $(d+4t-7, \mathcal{F})$-coloring $g$ that satisfies the claim. We claim that every component of $G[g^{-1}(\{i\})]$ for some $i \in [d+4t-7]$ is in $\mathcal{F}$. Let $C$ be a component of $G[g^{-1}(\{i\})]$ for some $i \in [d+4t-7]$. If $V(C) \cap Z \ne \emptyset$ then $V(C) \subseteq Z$ by (b), implying $C \in \mathcal{F}$ as $|V(C)| \leq |Z| \leq 4t-7$. If $V(C) \cap Z = \emptyset$ then $C$ is a component of $G'[g^{-1}(\{i\})]$, which implies $C \in \mathcal{F}$. Therefore, $g$ is a $(d+4t-7 , \mathcal{F})$-coloring of $G$ satisfying (a) and (b), contradicting our assumption.
\medskip

(2) \emph{For every separation $(A,B)$ of order at most $2t-3$, either $V(A) \setminus V(B) \subseteq Z$ or $V(B) \setminus V(A) \subseteq Z$.}

Suppose $G$ has a separation $(A,B)$ of order at most $2t-3$ such that both $V(A) \setminus V(B) \setminus Z$ and $V(B) \setminus V(A) \setminus Z$ are nonempty. Since $|Z| = |V(A) \cap Z| + |(V(B)\setminus V(A)) \cap Z|$, we may assume $|(V(B) \setminus V(A)) \cap Z| \leq \lfloor \frac{|Z|}{2} \rfloor \leq 2t-4$. Note that $V(B) \setminus V(A) \setminus Z \ne \emptyset$ implies that $|V(A) \cup Z| < |V(G)|$ and we can apply the claim on $A \cup G[Z]$ with $Z$ and $f$. Let $g_1$ be a $(d+4t-7, \mathcal{F})$-coloring of $A \cup G[Z]$ satisfying (a) and (b). Let $Z' = V(A \cap B) \cup (V(B) \cap Z)$. Since $|Z'| = |V(A \cap B)| + |(V(B) \setminus V(A)) \cap Z| \leq 4t-7$, we can apply the claim on $B$ with $Z'$ and $g_1|_{Z'}$. Let $g_2$ be a $(d+4t-7, \mathcal{F})$-coloring of $B$ satisfying (a) and (b).

Let $g$ be a coloring on $V(G)$ such that for each vertex $v$ of $G$, 
\[g(v) =
  \begin{cases}
    g_1(v)&\text{for }v \in V(A), \text{ and }\\
    g_2(v)&\text{for }v \in V(B).
  \end{cases}
\]
This is well defined since $g_1$ is identical to $g_2$ on $Z'$. We claim that $g$ is a $(d+4t-7 , \mathcal{F})$-coloring of $G$ satisfying (a) and (b), which contradicts our assumption. 

By the definition of $g_1$, it follows that $g(z) = g_1(z) = f(z)$ for every $z \in Z$. For every $vw \in E(G)$ with $v \in Z$ and $w \notin Z$, $g(v) \ne g(w)$ since $g_1 (v) = g(v) \ne g(w) = g_2(w)$ if $w \in V(A)$ and $g_2(v) = g(v) \ne g(w) = g_2(w)$ if $w \in V(B)$. This verifies (a) and (b).

Let $C$ be a component of $G[g^{-1}(\{i\})]$ for some $i \in [d+4t-7]$. If $V(C) \cap Z \ne \emptyset$ then $V(C) \cap (V(A) \setminus Z) = \emptyset$ by the definition of $g_1$ and (b), and $V(C) \cap (V(B) \setminus V(A) \setminus Z) = \emptyset$ by the definition of $g_2$ and (b). This implies $V(C) \subseteq Z$ and thus $C \in \mathcal{F}$ as $|V(C)|\leq |Z| \leq 4t-7$. If $V(C) \cap Z = \emptyset$ and $V(A) \cap V(B) \cap V(C) \ne \emptyset$ then $V(C) \subseteq Z' \setminus Z \subseteq V(A) \cap V(B)$ by the definition of $g_2$ and (b). Thus $C$ is a component of $G[g_1^{-1}(\{i\})]$, implying $C \in \mathcal{F}$. Finally, if $V(C) \cap Z = \emptyset$ and $V(A) \cap V(B) \cap V(C) = \emptyset$, then either $V(C) \subseteq V(A)\setminus V(B) \setminus Z$ or $V(C) \subseteq V(B) \setminus V(A) \setminus Z$, which implies that $C \in \mathcal{F}$ as $C$ is a component of either $G[g_1^{-1}(\{i\})]$ or $G[g_2^{-1}(\{i\})]$.
\medskip

(3) \emph{$G \setminus Z$ contains a bipartite $K_{2t-2} + I_{t}$ subdivision.}

Since $|Z| \leq 4t-7$, we may assume $f(Z) \subseteq \left \{d+1 , \dots , d+4t-7 \right \}$ by permuting colors. Suppose $G \setminus Z$ does not contain a bipartite $K_{2t-2} + I_{t}$ subdivision. 
Let $g_0$ be a $(d , \mathcal{F})$-coloring of $G\setminus Z$. Let $g : V(G) \to [d+4t-7]$ be a coloring such that for each vertex $v$ of $G$, 
\[g(v) =
  \begin{cases}
    g_1(v)& \text{for every }v \in V(G) \setminus Z \text{ and,}\\
    f(z)& \text{for every }z \in Z    .
  \end{cases}
\] 
We claim that $g$ is a $(d+4t-7 , \mathcal{F})$-coloring of $G$ satisfying (a) and (b), which contradicts our assumption. Let $C$ be a component of $G[g^{-1}(\{i\})]$ for some $i \in [d+4t-7]$. Since $g$ is identical to  $g_1$ on $V(G) \setminus Z$ and $g(u) \ne g(v)$ for every $u \in V(G) \setminus Z$ and $v \in Z$, $C$ is a component of either $G[g_1^{-1}(\{i\})]$ or $G[f^{-1}(\{i\})]$. This implies $C \in \mathcal{F}$. This proves (3).
\medskip

Since $G$ contains a bipartite $K_{2t-2} + I_{t}$ subdivision, Theorem~\ref{decomp} implies that there exists $X \subseteq V(G)$ with $|X| \leq 2t-4$ such that $G \setminus X$ admits a block decomposition with a bipartite block $U$ having at least $t+3$ vertices.
\medskip

(4) \emph{Every component of $G \setminus X \setminus V(U)$ is a subgraph of $G[Z]$.}

Let $C$ be a component of $G \setminus X \setminus V(U)$. Let $V_C$ be the set of vertices in $U$ adjacent to a vertex in $C$. As $U$ is a block and $C$ is a component of $G \setminus X \setminus V(U)$, it follows that $|V_C| \leq 1$. If $|V_C| = 1$ then let $v_C$ be the unique vertex in $V_C$. Let $A_C = G[V(C) \cup X \cup V_C]$ and $B_C = G \setminus V(C)$. Note that $V(A_C) \cap V(B_C) = X \cup V_C$ and $(A_C , B_C)$ is a separation of $G$ of order at most $2t-3$, since $|X|+|V_C| \leq 2t-3$. By (2), either $V(A_C) \setminus V(B_C)$ or $V(B_C) \setminus V(A_C)$ is in $Z$. Since $Z$ is stable, $V(U) \setminus V_C \subseteq V(B_C) \setminus V(A_C)$ and $U$ is 2-connected as $|V(U)| \geq t+3$, $V(B_C) \setminus V(A_C)$ is not a subset $Z$. Therefore, $V(A_C) \setminus V(B_C) = V(C)$ is a subset of $Z$. This proves (4).
\medskip

Since $U \setminus Z$ is a bipartite subgraph of $G$, let $\left \{X_1 , X_2 \right \}$ be its bipartition. By (4), it follows that $V(G) = Z \cup (X \setminus Z) \cup X_1 \cup X_2$. Let us choose three colors $\left \{c_1 , c_2 , c_3 \right \} \subseteq [4t-4] \setminus f(Z)$. Let $g : V(G) \to [4t-4] \subseteq [d+4t-7]$ be a coloring defined as follows: 
\[g(x) =
  \begin{cases}
f(x)& \text{for }x \in Z, \\
c_1& \text{for }x \in X \setminus Z,\\
c_2&\text{if }x \in X_1, \\
c_3&\text{if }x \in X_2.
  \end{cases}
\]
We claim that $g$ is a $(d+4t-7 , \mathcal{F})$-coloring of $G$ satisfying (a) and (b), which contradicts our assumption.

Let $C$ be a component of $G[g^{-1}(\{i\})]$ for some $i \in [d+4t-7]$. One of the following cases hold: either $V(C) \subseteq Z$ or $V(C) \subseteq X \setminus Z$ or $V(C) \subseteq X_1$ or $V(C) \subseteq X_2$. Since $|Z|$ and $|X|$ are at most $4t-7$ and both $X_1$ and $X_2$ are stable in $G$, $C$ is in $\mathcal{F}$.
\end{proof}

Now we present proofs of our main theorems.

\begin{proof}[Proof of Theorem~\ref{mainthm}]
Let $\mathcal{F}$ be the set of graphs of maximum degree at most $\max(N(2t-2 , t) , 4t-8)$ where $N$ is in Corollary~\ref{cor:genbipweakhad}. Corollary~\ref{cor:genbipweakhad} implies that every graph with no bipartite $K_{2t-2} + I_{t}$ subdivision has a $(2t-2, \mathcal{F})$-coloring. By Lemma~\ref{precoloring}, $G$ has a $(6t-9, \mathcal{F})$-coloring, implying that $G$ is $(6t-9)$-colorable with defect $\max(N(2t-2 , t) , 4t-8)$.
\end{proof}

\begin{proof}[Proof of Theorem~\ref{mainthm2}]
Let $u(t) := C(t , N(2t-2 , t))$ where $C$ is in Theorem~\ref{partition} and $N$ is in Corollary~\ref{cor:genbipweakhad}. Let $\mathcal{F}$ be the set of graphs that every component has at most $\max(u(t) , 4t-7)$ vertices. By Corollary~\ref{cor:genbipweakhad} and Theorem~\ref{partition}, every graph with no odd $K_t$ minor and no bipartite $K_{2t-2} + I_{t}$ subdivision has a $(3(2t-2), \mathcal{F})$-coloring. By Lemma~\ref{precoloring}, $G$ has a $(10t-13, \mathcal{F})$-coloring, implying that $G$ is $(10t-13)$-colorable with clustering $\max(u(t) , 4t-7)$.
\end{proof}

\section{Concluding Remarks}\label{sec:conclude}
\subsection{List-coloring variant}
We may consider a list-coloring variant of defective coloring. For integers $s,N \geq 0$, a graph $G$ is \emph{$s$-choosable} with \emph{defect $N$} if for every set of lists $\left \{ L_v \right \}_{v \in V(G)}$ with $|L_v| \geq s$ for every $v \in V(G)$, there is a map $f : V(G) \to \bigcup_{v \in V(G)}{L_v}$ with $f(v) \in L_v$ for each $v \in V(G)$ such that $G[f^{-1}(\left\{i \right \})]$ has maximum degree at most $N$ for every $i \in \bigcup_{v \in V(G)}{L_v}$.

As we remarked in Section~\ref{sec:intro}, Theorems~\ref{thm:weakhad} and~\ref{thm:genweakhad} can be extended for list-colorings. For instance, Ossona de Mendez, Oum, and Wood~\cite{ossonademendez2016} showed that for integers $s,t \geq 1$ and every graph $G$ with no $K_{s,t}^*$ subgraph, there is $N=N(s,t)$ such that $G$ is $s$-choosable with defect $N$. It follows that for $t \geq 1$, every graph with no $K_{t+1}$ minor is $t$-choosable with defect $M$ for some constant $M=M(t)$, which is also implied by the proof of Edwards, Kang, Kim, Oum, and Seymour~\cite{EKKOS2014}.

Note that every $n$-vertex graph with no $K_t$ minor contains $O(t \sqrt{\log t}\,n)$ edges~\cite{Kostochka1982, Kostochka1984,Thomason1984, Thomason2001}. In contrast to graphs with no $K_t$ minor, an $n$-vertex graph with no odd $K_3$ minor may contain $\Omega(n^2)$ edges. For example, complete bipartite graphs have no odd $K_3$ minor. 

\begin{THM}[Kang~\cite{kang2013improper}]\label{nolistvariant}
For each integer $N \geq 0$, there is a function $s = s(d) = (1/2 + o(1))\log_2 d$ as $d \to \infty$ such that if a graph $G$ has minimum degree at least $d$, $G$ is not $s$-choosable with defect $N$.
\end{THM}

By Theorem~\ref{nolistvariant}, it follows that for integers $t \geq 1$ and $s, N \geq 0$, there are graphs with no odd $K_t$ minor not $s$-choosable with defect $N$.

\subsection{Extending Theorems~\ref{mainthm} and~\ref{mainthm2}} We extend our main results to a slightly larger class of graphs. As we mentioned in Section~\ref{sec:term}, $G$ contains $H$ as an odd minor if and only if a signed graph $(G,E(G))$ contains a signed graph $(H, E(H))$ as a minor. We review the concepts of signed graphs and their minors in Appendix~\ref{app}. 

Given $\Sigma \subseteq E(H)$, we provide alternative characterization for signed graphs $(G,E(G))$ containing $(H , \Sigma)$ as a minor. A signed graph $(G,E(G))$ contains a signed graph $(H , \Sigma)$ as a \emph{minor} if and only if 
\begin{enumerate}
\item there exist vertex-disjoint subgraphs $\left \{T_u \right \}_{u \in V(H)}$ in $G$ which are trees, and 
\item a coloring $\alpha : \bigcup_{u \in V(H)}{V(T_u)} \to \left \{1,2 \right \}$ such that for every $u \in V(H)$, every edge in $T_u$ is bichromatic, and for every edge $vw \in E(H)$, there is an edge $e \in E(G)$ that joins $V(T_v)$ and $V(T_w)$ where $e$ is monochromatic if and only if $vw \in \Sigma$.
\end{enumerate}

Note that for every $\Sigma \subseteq E(K_t)$, a signed graph $(K_{2t} , E(K_{2t}))$ contains $(K_t , \Sigma)$ as a minor. Replacing $t$ by $2t$, Theorem~\ref{mainthm} implies that for every $t \geq 2$ and every $\Sigma \subseteq E(K_t)$, if $(G,E(G))$ contains no $(K_t , \Sigma)$ as a minor, then $G$ is $(12t-9)$-colorable with defect $s(2t)$. Theorem~\ref{mainthm2} also implies that for every $t \geq 2$ and every $\Sigma \subseteq E(K_t)$, if $(G,E(G))$ contains no $(K_t , \Sigma)$ as a minor, then $G$ is $(20t-13)$-colorable with clustering $C(2t)$.

By modifying the proofs in Section~\ref{sec:structure}, we can improve these bounds further. In the proof of Lemma~\ref{force}, we join $V(M_i)$ and $V(M_j)$ with a parity-breaking path with respect to $\alpha$ for $1 \leq i < j \leq t$. Because $\alpha(x_i) = \beta(x_i)$ and $\alpha(y_i) \ne \beta(y_i)$, we can also join $V(M_i)$ and $V(M_j)$ with a path that is not parity-breaking with respect to $\alpha$. In particular, Lemma~\ref{force} forces not only an odd $K_t$ minor, but also a signed $(K_t , \Sigma)$ minor for \emph{every} $\Sigma \subseteq E(K_t)$. This extends Theorems~\ref{mainthm} and~\ref{mainthm2} as follows.

\begin{COR}\label{mainthm_2}
For each integer $t \geq 2$, there exists an integer $s = s(t)$ such that for every $\Sigma \subseteq E(K_t)$ and $(G,E(G))$ with no $(K_t , \Sigma)$ minor, the graph $G$ is $(6t-9)$-colorable with defect~$s$.
\end{COR}

\begin{COR}\label{mainthm2_2}
For each integer $t \geq 2$, there exists an integer $C = C(t)$ such that for every $\Sigma \subseteq E(K_t)$ and $(G,E(G))$ with no $(K_t , \Sigma)$ minor, the graph $G$ is $(10t-13)$-colorable with clustering $C$.
\end{COR}

\subsection{Upper bound of maximum degree}
In Theorem~\ref{thm:genweakhad}, the function $M(s,t,\delta_1 , \delta_2)$ is defined as follows.
\[ M(s,t,\delta_1,\delta_2) = \left\{ \begin{array}{lll}
								         t-1 & \mbox{if $s=1$}\\
								         \frac{\delta_2 t(\delta_1 - 2)}{2} + \delta_1 & \mbox{if $s=2$}\\
								         (\delta_1 - s)\left ( \binom{\lfloor \delta_2 \rfloor}{s-1}(t-1) + \frac{\delta_2}{2} \right ) + \delta_1 & \mbox{if $s>2$}
								        \end{array} \right. \]
								        
Therefore, it follows that $N(2t-2,t) = M(2t-2,t,O(t^2),O(t^2)) = \exp(O(t \log t))$ in Corollary~\ref{cor:genbipweakhad}. Hence we have the upper bound of $s(t) = \exp(O(t \log t))$ in Theorem~\ref{mainthm}.

If one replaces a bipartite $K_{2t-2} + I_t$ subdivision of Lemma~\ref{force} with a bipartite $K_{3t-2}$ subdivision, one may set $s(t) = O(t^4)$ since $N(3t-3,1) = O(t^4)$. However, this will increase the number of colors from $6t-9$ to $7t-10$ of Theorem~\ref{mainthm}, as graphs with no bipartite $K_{3t-2}$ subdivision are defectively colored with $3t-3$ colors, which is more than $2t-2$ colors in defective coloring of graphs with no bipartite $K_{2t-2} + I_t$ subdivision.

\appendix
\section{signed graphs} \label{app}
We review elementary concepts of signed graphs, following definitions in~\cite{harary1953notion} and~\cite{huynh2012even} unless stated otherwise.

A \emph{signed graph} $(G,\Sigma)$ is a graph $G=(V,E)$ equipped with a \emph{signature} $\Sigma \subseteq E$. To avoid confusion, graphs always denote unsigned graphs. Every signed graph is assumed to be simple; parallel edges and loops are not allowed. If an edge $e \in E(G)$ is in $\Sigma$, $e$ is \emph{negative}. Otherwise, $e$ is \emph{positive}. 

For two sets $A$ and $B$, $A \Delta B$ denotes the set $(A \setminus B) \cup (B \setminus A)$. For a graph $G$ and $X \subseteq V(G)$, let $\delta_G(X)$ be the set of edges joining $X$ and $V(G) \setminus X$. For a signature $\Sigma$, a \emph{re-signing} on $X$ is an operation that replaces $\Sigma$ with another signature $\Sigma \Delta \delta_G (X)$ for some $X \subseteq V(G)$. Note that for $X,Y \subseteq V(G)$, applying re-signing on $X$ and $Y$ is identical to applying re-signing on $X \Delta Y$. In particular, \emph{re-signing at $v$} is the operation that replaces $\Sigma$ with $\Sigma \Delta \delta_G (\left \{ v \right \})$. Note that applying a re-signing operation on $X$ is identical to applying re-signing operations at all vertices in $X$.

Two signatures $\Sigma$ and $\Sigma'$ are \emph{equivalent} if $\Sigma'$ can be obtained from $\Sigma$ by re-signing operations; $\Sigma$ and $\Sigma'$ are equivalent if and only if there is $X \subseteq V(G)$ such that $\Sigma' = \Sigma \Delta \delta_G(X)$. Two signed graphs $(G,\Sigma)$ and $(G,\Sigma')$ are \emph{equivalent} if $\Sigma$ is equivalent to $\Sigma'$.

A cycle $C$ is called \emph{balanced} if it contains an even number of negative edges. Two signed graphs $(G , \Sigma)$ and $(G , \Sigma')$ have the same set of balanced cycles if and only if $\Sigma$ and $\Sigma'$ are equivalent (see~\cite{harary1953notion}).

A map $f : V(G) \to V(H)$ is an \emph{isomorphism} from $(G,\Sigma)$ and $(H,\Sigma')$ if $f$ is an isomorphism from $G$ to $H$, and $uv \in \Sigma$ if and only if $f(u)f(v) \in \Sigma'$. If there is an isomorphism from $(G,\Sigma)$ to $(H,\Sigma')$, $(G,\Sigma)$ is \emph{isomorphic} to $(H,\Sigma')$.

For two signed graphs $(G,\Sigma)$ and $(H,\Sigma')$, $(H, \Sigma')$ is a \emph{minor} of $(G, \Sigma)$ if a signed graph isomorphic to $(H , \Sigma')$ can be obtained from $(G , \Sigma)$ by deleting vertices, deleting edges, applying re-signing operations, and contracting positive edges.

To avoid parallel edges, if we contract a positive edge $uv$ such that there exists a vertex $w \notin \left \{u,v \right \}$ and edges $wu, wv \in E(G)$ of different signs, then we should remove either $wu$ or $wv$ before contracting $uv$.

When applying a series of operations to find minors, we may assume that deleting vertices and edges always precede re-signing operations; contracting a positive edge $uv$ into a new vertex $t$ and re-signing at $t$ is identical to re-signing on $\left \{u,v \right \}$ and contracting a positive edge $uv$ into a vertex $t$. This implies the following, which can be found in~\cite{geelen2004colouring}.

\begin{LEM}\label{parityminor}
For graphs $G$, $H$ and a signature $\Sigma \subseteq E(H)$, a signed graph $(G , E(G))$ contains a signed graph $(H , \Sigma)$ as a minor if and only if
\begin{enumerate}
\item there are vertex-disjoint subgraphs $\left \{T_u \right \}_{u \in V(H)}$ of $G$ assigned to vertices in $V(H)$,
\item for every $u \in V(H)$, $T_u$ is a tree and has a proper $2$-coloring  $c_u : V(T_u ) \to \left \{1,2 \right \}$, and
\item for every edge $uv \in E(H)$, there is an edge $e = ab \in E(G)$ that joins $T_u$ and $T_v$ such that $c_u (a) = c_v (b)$ if and only if $uv \in \Sigma$.
\end{enumerate}
\end{LEM}
\end{document}